\documentclass[preprint]{elsarticle}
\usepackage{amsfonts}
\usepackage{amsmath}
\usepackage{amssymb}

\newtheorem{theorem}{Theorem}

\newtheorem{corollary}[theorem]{Corollary}

\newtheorem{definition}[theorem]{Definition}

\newtheorem{proposition}[theorem]{Proposition}
\newtheorem{remark}[theorem]{Remark}

\newenvironment{proof}[1][Proof]{\noindent \textbf{#1.} }{\  \rule{0.5em}{0.5em}}
\begin{document}

\begin{frontmatter}
	
	\title{Commutativity and spectral properties for a general class of 
	Sz\'asz-Mirakjan-Durrmeyer operators}

	\author[1]{Ulrich Abel} 
	\author[2]{Ana Maria Acu}
	\author[3]{Margareta Heilmann}
	\author[4]{Ioan Ra\c sa}

	\vspace{10cm}

	\address[1] {Fachbereich MND
		Technische Hochschule Mittelhessen, Germany, 
		e-mail: ulrich.abel@mnd.thm.de }
	\address[2]{Department of Mathematics and Informatics, Lucian Blaga University of Sibiu,  Romania ({\it Corresponding Author}), e-mail: anamaria.acu@ulbsibiu.ro}
	\address[3]{School of Mathematics and Natural Sciences, University of Wuppertal,  Germany, e-mail: heilmann@math.uni-wuppertal.de}
	\address[4]{Technical University of Cluj-Napoca, Faculty of Automation and Computer Science, Department of Mathematics, Str. Memorandumului nr. 28, 400114 Cluj-Napoca, Romania
		e-mail:  ioan.rasa@math.utcluj.ro }
	
	\begin{abstract} 	
		{In this paper we present commutativity results for a general class of Sz\'asz-Mirakjan-Durrmeyer type operators and associated differential operators and investigate their eigenfunctions.
		} 
	\end{abstract}
	
	
\end{frontmatter}

\fbox{\fbox{{\large \jobname.tex}}} \hskip1cm ~ {\large \today}


\section{Introduction}

For fixed $j \in \mathbb{Z}$ we consider sequences of positive linear
operators $S_{n,j}$ which can be viewed as a generalization of the
Sz\'asz-Mirakjan-Durrmeyer operators \cite{MaTo1985}, Phillips operators 
\cite{Ph1954} and corresponding Kantorovich modifications of higher order.

For $j\in {\mathbb{N}}$, these operators possess the exceptional
property to preserve constants and the monomial $x^{j}$.

It turned out (see \cite{Snj2}), that an extension of this family covers certain well-known
operators studied before, so that the outcoming results could be unified.
Continuous functions  of exponential growth on the positive half-axis are approximated  
by the operators $S_{n,j}$, $j \in \mathbb{Z}$   in each point of
continuity (see \cite[Corollary 12]{Snj2}.

In this paper we investigate commutativity properties of the operators and their
commutativity with certain differential operators as well as their spectral properties.

For $A\geq 0$, we denote by $E_{A}$ the space of functions $f:[0,\infty
)\longrightarrow \mathbb{R}$, $f$ locally integrable, satisfying the growth
condition $\left\vert f\left( t\right) \right\vert \leq Ke^{At}$, $t\geq 0$
for some positive constant $K$. Furthermore, define $E:=\displaystyle%
\bigcup_{A\geq 0}E_{A}$.

Throughout the paper we use the convention that a sum is zero if the upper
limit is smaller than the lower one.


\begin{definition}
Let $f \in E_A$, $n>A$ and $j\in{\mathbb{Z}}$. Then the operators $S_{n,j}$
are defined by 
\begin{equation*}
(S_{n,j}f)(x)=f(0)\sum_{k=0}^{j-1}s_{n,k}(x)+\sum_{k=j}^{\infty }s_{n,k}(x) n
\int_{0}^{\infty }s_{n,k-j}(t)f(t)dt
\end{equation*}
where 
\begin{equation*}
s_{n,k}(x)=\dfrac{(nx)^k}{k!}e^{-nx}, \, k\geq 0.
\end{equation*}
For the sake of simplicity we define $s_{n,k}(x)=0$, if $k<0$.
\end{definition}

For $j \leq 0$ we have 
\begin{equation*}
(S_{n,j}f)(x)=\sum_{k=0}^{\infty }s_{n,k}(x) n\int_{0}^{\infty
}s_{n,k-j}(t)f(t)dt.
\end{equation*}
For the special case $j=0$ we get the Sz\'asz-Mirakjan-Durrmeyer operators
first defined in \cite{MaTo1985}, for $j=1$ the Phillips operators, also
called genuine Sz\'asz-Mirakjan-Durrmeyer operators, \cite{Ph1954}. \ For $j
\leq -1$ the operators coincide with certain auxiliary operators, see, e.g., 
\cite[(1.4)]{HeiMue1989} for $c=0$, named $M_{n,r}$ there, $r=-j$ or \cite[%
(3.5)]{He1992} for $c=0$, named $M_{n,s}$ there, $s=-j$. They can also be
considered as corresponding Kantorovich modifications of higher order. 

In \cite[Lemma 2]{Snj2} it was proved that $S_{n,j}f\in E_{2A}$ for $n>2A$, $A\geq 0$, $f\in E_A$.

For $l=0,1,\ldots$ we denote $e_l(t)=t^l$, $t\geq0$.

All these operators preserve $e_0$ and for $j\geq 1$ also $e_j$ since 
\begin{eqnarray*}
(S_{n,j} e_j)(x) & = & \sum_{k=j}^{\infty} s_{n,k} (x){n \int_0^\infty
s_{n,k-j} (t) t^j dt} \\
& = & x^j \sum_{k=j}^{\infty} s_{n,k-j} (x) = x^j.
\end{eqnarray*}
Furthermore, for $j \geq 1$ it is easy to see that $S_{n,j}(f;0)=f(0)$.

Throughout the paper we use the following formulas 
\begin{eqnarray}
\label{eq-der1}
	s_{n,k}^{\prime }(x) 
	& = & 
	n (s_{n,k-1}(x) -s_{n,k}(x) ),   
\\
\label{eq-int1}
	\int_0^\infty s_{n,l} (t) t^r dt 
	& = & 
	\frac{(l+r)!}{l!} n^{-r-1},
\\
\label{eq-int2}
	\int_0^\infty s_{m,r} (t) s_{n,l} (t) dt 
	& = & 
	\frac{m^r n^l}{(m+n)^{r+l+1}}
	\frac{(r+l)!}{r!l!}.
\end{eqnarray}

\section{Commutativity of the operators}

In this section our aim is to show that the operators commute for a fixed $j$%
. For $j=0$ this was proved in \cite{He1988}. Indeed, here we prove a more
general result for compositions of the operators from which the
commutativity then follows as a corollary as well as a nice representation
for iterates of $S_{n,j}$. For $j=0$ the following result was proved in \cite%
{AbIv2005}, for $j=1$ see \cite[Theorem 3.1]{Heilmann-Tachev-Philipps-2011}
and for $j\leq 0$ see \cite[Theorem 1]{Sw1}.

Due to \cite[Remark 3]{Snj2} we obtain
\begin{equation*}
\left\vert S_{m,j}\left( S_{n,j}f\right) \left( x\right) \right\vert \leq
1+\left( \frac{n-A}{n}\right)^{j-1}\left( \frac{m-\frac{An}{n-A}}{m}\right)
^{j-1}\exp \left( \frac{Am}{m-\frac{An}{n-A}}x\right) 
\end{equation*}
for each integer $j$, $n>A$ and $m > \frac{An}{n-A}$.

For $n>A$ the inequality $m > \frac{An}{n-A}$ is equivalent to 
$ \frac{mn}{m+n}>A$.

\begin{theorem}
\label{th-composite} Let $f\in E_A$, $n > A$ and  $m > \frac{An}{n-A}$. Then 
\begin{equation}
\label{comp1}
S_{m,j}(S_{n,j}f)=S_{\frac{mn}{m+n},j}f.
\end{equation}
\end{theorem}

\begin{proof}
Due to the preceding remarks we only have to prove the theorem for $j \geq 2$.
The considerations above show that both sides of (\ref{comp1}) are well defined.
According to \cite[Lemma 2]{Snj2}, $S_{m,j}$ can be applied
to $S_{n,j}f$.

Now we can write 
\begin{eqnarray*}
\lefteqn{(S_{\frac{mn}{m+n},j}f)(x)} \\
&=&f(0)\sum_{k=0}^{j-1}s_{\frac{mn}{m+n},k}(x)+\sum_{k=j}^{\infty }s_{\frac{%
mn}{m+n},k}(x)\cdot \frac{nm}{n+m}\cdot \int_{0}^{\infty }s_{\frac{mn}{m+n}%
,k-j}(t)f(t)dt \\
&=: &T_{1}+T_{2}
\end{eqnarray*}%
\begin{eqnarray*}
(S_{n,j}f)(x) &=&f(0)\sum_{k=0}^{j-1}s_{n,k}(x)+\sum_{k=j}^{\infty
}s_{n,k}(x)\cdot n\cdot \int_{0}^{\infty }s_{n,k-j}(t)f(t)dt \\
&=: &(S_{n,j,1}f)(x)+(S_{n,j,2}f)(x)
\end{eqnarray*}%
\begin{eqnarray*}
(S_{m,j}(S_{n,j}f)(y))(x)
&=&(S_{m,j,1}(S_{n,j,1}f)(y))(x)+(S_{m,j,2}(S_{n,j,1}f)(y))(x) \\
&&+(S_{m,j,1}(S_{n,j,2}f)(y))(x)+(S_{m,j,2}(S_{n,j,2}f)(y))(x) \\
&=: &T_{3}+T_{4}+T_{5}+T_{6} .
\end{eqnarray*}
It is easy to verify that $T_5 =0$. We will now prove that $T_6=T_2$ and $%
T_3+T_4=T_1$. 

With (\ref{eq-int2}) we get
\begin{eqnarray}  \label{T6}
T_{6} &=&\sum_{k=j}^{\infty }s_{m,k}(x) m \int_{0}^{\infty }s_{m,k-j}(t) 
\left[ \sum_{l=j}^{\infty }s_{n,l}(t) n \int_{0}^{\infty }s_{n,l-j}(y)f(y)dy%
\right] dt  \notag \\
&=&mn\int_{0}^{\infty }\left\{ f(y)\sum_{k=j}^{\infty }\sum_{l=j}^{\infty
}s_{m,k}(x)s_{n,l-j}(y){\int_{0}^{\infty }s_{m,k-j}(t)s_{n,l}(t)dt}\right\}
dy  \notag \\
&=: &\int_{0}^{\infty }f(y)T_{m,n}(x,y)dy .
\end{eqnarray}%
\begin{eqnarray*}
T_{m,n}(x,y) &=& mn\sum_{k=0}^{\infty }
\sum_{l=0}^{\infty}s_{m,k+j}(x)s_{n,l}(y)\frac{m^{k}n^{l+j}}{(m+n)^{k+l+1+j}}
\cdot \frac{(k+l+j)!}{k!(l+j)!} \\
&=&\frac{mn}{m+n}e^{-mx}e^{-ny} \sum_{k=0}^{\infty }\frac{m^{2k+j}x^{k+j}}{%
(m+n)^{k+j}}\frac{1}{k!(k+j)!} T(y) ,
\end{eqnarray*}
where
\begin{eqnarray*}
	T(y)
	& :=& 
	\sum_{l=0}^{\infty }\frac{n^{2l+j}y^{l}}{(m+n)^{l}} \frac{ (k+l+j)!}{l!(l+j)!} 
\\
	& = &
	n^{j}\sum_{l=0}^{\infty }\frac{(\frac{n^{2}}{m+n}y)^{l}}{l!(l+j)!}\cdot (k+l+j)! .
\end{eqnarray*}
Using the relation 
\begin{eqnarray*}
\sum_{l=0}^{\infty }\frac{(l+k+j)!}{(l+j)!}\frac{z^{l}}{l!} &=& z^{-j}\left( 
\frac{d}{dz}\right) ^{k}\sum_{l=0}^{\infty }\frac{z^{l+k+j}}{l!} \\
&=& z^{-j}\left( \frac{d}{dz}\right) ^{k}\left( z^{k+j}e^{z}\right) \\
&=& e^{z}\sum_{\nu =0}^{k}{\binom{k}{\nu }}\frac{(k+j)!}{(\nu +j)!}z^{\nu }
\end{eqnarray*}
with $z=\frac{n^{2}}{m+n}y$ we obtain 
\begin{equation*}
T(y)=n^{j}e^{\frac{n^{2}}{m+n}y}\sum_{\nu =0}^{k}{\binom{k}{%
\nu }}\frac{(k+j)!}{(\nu +j)!}\left( \frac{n^{2}}{m+n}y\right) ^{\nu }.
\end{equation*}%
Inserting this in $T_{m,n}(x,y)$ above we get 
\begin{align*}
	\lefteqn{T_{m,n}(x,y)} 
\\
	&=
	\frac{n^{j+1}m}{m+n}e^{-mx}e^{-ny}e^{\frac{n^{2}}{m+n}y}
	\sum_{k=0}^{\infty} \frac{m^{2k+j}x^{k+j}}{(m+n)^{k+j}}
	\frac{1}{k!(k+j)!}\sum_{\nu =0}^{k}{\binom{k}{\nu }}
	\frac{(k+j)!}{(\nu +j)!}\left( \frac{n^{2}}{m+n}y\right) ^{\nu } 
\\
	&=
	\frac{n^{j+1}m}{m+n}e^{-mx}e^{-ny}e^{\frac{n^{2}}{m+n}y}
	\sum_{\nu=0}^{\infty }\frac{1}{\nu !(\nu +j)!}
	\left( \frac{n^{2}}{m+n}y\right)^{\nu}
	\sum_{k=\nu }^{\infty } \frac{1}{(k-\nu )!} \frac{m^{2k+j}x^{k+j}}{(m+n)^{k+j}}
\end{align*}%
Now we calculate 
\begin{eqnarray*}
	\sum_{k=\nu }^{\infty } \frac{1}{(k-\nu )!} 
	\frac{m^{2k+j}x^{k+j}}{(m+n)^{k+j}}
	&=&
	\frac{m^{2\nu +j}x^{\nu +j}}{(m+n)^{\nu +j}}
	\sum_{k=0}^{\infty }\frac{\left( \frac{m^{2}x}{m+n}\right) ^{k}}{k!} 
\\
	&=&
	\frac{m^{2\nu +j}x^{\nu +j}}{(m+n)^{\nu +j}}e^{\frac{m^{2}}{m+n}x}
\end{eqnarray*}%
which gives 
\begin{eqnarray*}
T_{m,n}(x,y) &=&\frac{nm}{m+n}e^{-\frac{mn}{m+n}x}e^{-\frac{mn}{m+n}%
y}\sum_{\nu =0}^{\infty }\frac{\left( \frac{mn}{m+n}x\right) ^{\nu +j}}{(\nu
+j)!}\frac{\left( \frac{mn}{m+n}y\right) ^{\nu }}{\nu !} \\
&=&\frac{mn}{m+n}\sum_{\nu =0}^{\infty }s_{\frac{mn}{n+m},\nu +j}(x)s_{\frac{%
mn}{n+m},\nu }(y).
\end{eqnarray*}%
Thus with (\ref{T6}) 
\begin{equation*}
T_{6}=T_{2}
\end{equation*}%
and it remains to prove that 
\begin{equation*}
T_{3}+T_{4}=T_{1} .
\end{equation*}%
We have 
\begin{eqnarray*}
T_{3} &=&f(0)\sum_{l=0}^{j-1}s_{n,l}(0)\sum_{k=0}^{j-1}s_{m,k}(x) \\
&=&f(0)\sum_{k=0}^{j-1}s_{m,k}(x) 
\end{eqnarray*}
and by using (\ref{eq-int2})
\begin{eqnarray*}
T_{4} &=&\sum_{k=j}^{\infty }s_{m,k}(x)\cdot m\cdot \int_{0}^{\infty
}s_{m,k-j}(t)\left[ f(0)\sum_{l=0}^{j-1}s_{n,l}(t)\right] dt \\
&=&f(0)m\sum_{k=j}^{\infty }s_{m,k}(x)\sum_{l=0}^{j-1} \int_{0}^{\infty
}s_{m,k-j}(t)s_{n,l}(t)dt \\
&=&f(0)m\sum_{k=j}^{\infty }s_{m,k}(x)\sum_{l=0}^{j-1}\frac{(k+l-j)!}{%
(k-j)!l!}\cdot \frac{m^{k-j}n^{l}}{(m+n)^{k+l-j+1}} \\
&=&f(0)\frac{m}{m+n}e^{-mx}\sum_{l=0}^{j-1}\frac{1}{l!}\left( \frac{n}{m+n}%
\right) ^{l}
\left( \frac{m}{m+n}\right) ^{-j}\sum_{k=j}^{\infty }\frac{(mx)^{k}}{k!}\frac{(k+l-j)!}{(k-j)!}
\cdot \left( \frac{m}{m+n}\right) ^{k} .
\end{eqnarray*}
With Leibniz formula for differentiation we can write 
\begin{eqnarray*}
	\sum_{k=j}^{\infty }\frac{z^{k}}{k!}\frac{(k+l-j)!}{(k-j)!}
	& = &
	z^j \sum_{k=j}^{\infty }\frac{1}{k!} \left ( z^{k+l-j} \right )^{(l)}
\\
	& = &
	z^j \frac{d^{l}}{dz^{l}}\left\{ z^{l-j}\left[
	e^{z}-\sum_{k=0}^{j-1}\frac{z^{k}}{k!}\right] \right\} 
\\
	&=&
	z^j \sum_{i=0}^{l}{\binom{l}{i}} (l-i)^{\underline{l-i}} z^{i-j}   
	\left [e^{z}-\sum_{k=i}^{j-1} \frac{z^{k-i}}{(k-i)!}\right] 
\\
	&=&
	\sum_{i=0}^{l}{\binom{l}{i}} (-1)^{l-i} (j-i-1)^{\underline{l-i}} z^{i}   
	\left [e^{z}-\sum_{k=i}^{j-1} \frac{z^{k-i}}{(k-i)!}\right] .	
\end{eqnarray*}%
From this equation by setting $z=\frac{m^2x}{m+n}$ we get after interchanging
the order of summation 
\begin{align*}
	T_{4} 
	&=
	f(0) \left (\frac{m}{m+n} \right )^{1-j} e^{-mx}
	\sum_{i=0}^{j-1}\frac{1}{i!}\left( \frac{m^2x}{m+n}\right)^{i}
	\sum_{l=i}^{j-1}{\binom{j-i-1}{l-i}}  (-1)^{l-i} \left( \frac{n}{m+n}\right) ^{l}
\\
	& \qquad \qquad \qquad \times \left\{ e^{\frac{m^{2}x}{m+n}} - \sum_{k=i}^{j-1} 
	\frac{\left( \frac{m^2x}{m+n}\right)^{k-i}}{(k-i)!} \right \} .
\end{align*}%
With 
\begin{equation*}
\sum_{l=i}^{j-1}{\binom{j-i-1}{l-i}}(-1)^{l-i}\left( \frac{n}{m+n}\right)
^{l}=\left( \frac{n}{m+n}\right) ^{i}\left( \frac{m}{m+n}\right) ^{j-i-1}
\end{equation*}
we derive
\begin{eqnarray*}
	T_{4} 
	&=&
	f(0)e^{-mx}\sum_{i=0}^{j-1}\frac{\left (\frac{m^2x}{m+n} \right )^i}{i!}
	\left(\frac{n}{m} \right )^i
	\left \{e^{\frac{m^{2}x}{m+n}} -\sum_{k=i}^{j-1} 
	\frac{\left( \frac{m^2x}{m+n}\right)^{k-i}}{(k-i)!} \right \}
\\
	&=& 
	f(0)\left\{ \sum_{i=0}^{j-1}s_{\frac{nm}{n+m},i}(x)
	-e^{-mx}\sum_{k=0}^{j-1} \frac{\left (\frac{m^2x}{m+n} \right )^k}{k!}
	\sum_{i=0}^{k}{\binom{k}{i}}\left( \frac{n}{m} \right)^{i}\right\} 
\\
	& = &
	f(0)\left\{ \sum_{i=0}^{j-1}s_{\frac{nm}{n+m},i}(x)	 - \sum_{k=0}^{j-1} s_{m,k}(x) \right \}
\\
	&=&
	T_{1}-T_{3}.
\end{eqnarray*}
\end{proof}

Now we get the desired commutativity of the operators and the representation
for their iterates as a corollary.

\begin{corollary}
\label{cor1} For all $f \in E_A$, $n,m > 2A$ we have 
\begin{equation*}  
S_{m,j}({S}_{n,j}f)={S}_{n,j}({S}_{m,j}f)
\end{equation*}
and for $l \in \mathbb{N}$, $m > lA$, 
\begin{equation}  \label{commut-3}
S_{m,j}^lf=S_{\frac{m}{l},j}f .
\end{equation}
\end{corollary}
\begin{remark}
From (\ref{commut-3}) it is easy to see that
\begin{equation*}
I-\left( I-S_{n,j}\right) ^{m+1} =\sum_{i=0}^{m}(-1)^{i}{\binom{m+1}{i+1}%
}S_{n,j}^{i+1} 
=\sum_{i=0}^{m}(-1)^{i}{\binom{m+1}{i+1}}S_{\frac{n}{i+1},j} ,
\end{equation*}
i.e., iterative combinations are linear combinations. For $j=1$ see \cite{Heilmann-Tachev-Phillips-2013}.
\end{remark}
\section{Derivatives and differential operators}

As already known for $j=0$ and $j=1$ the operators $S_{n,j}$ are strongly
connected to appropriate differential operators which can be used to state
asymptotic relations in a concise form. In Theorem \ref{lem-com-diff} we
will show that $S_{n,j}$ commutes with the corresponding differential
operator, already defined in \cite[(3)]{Snj2} for sufficiently smooth functions. This property was first proved
for $j=0$ in \cite[Lemma 3.1]{He1992} and for $j=1$ in \cite[Theorem 3.2]%
{Heilmann-Tachev-Philipps-2011}. For $j \leq 1$ see also \cite[Section 5]%
{Hei2015} and the references therein.

For the proof of Theorem \ref{lem-com-diff} we need the following
representation for the derivatives of $S_{n,j}f$ which can be found for $j=0$
in \cite[(3.11)]{Hei1989} or \cite[(3.1)]{He1992}. For $j\leq 0$ see \cite[%
Lemma 4]{Sw1}.

\begin{proposition}
Let $l\in {\mathbb{N}}$ and $f^{(l)}\in E_A$, $A>0$ and $n>A$. Then \label%
{prop-derivative} 
\begin{eqnarray*}
(S_{n,j} f)^{(l)}(x) & = & n \sum_{\nu =1}^{\min{\{j-1,l-1\}}}
s_{n,j-1-\nu}^{(l-1-\nu)} (x) f^{(\nu)} (0) \\
& & + n \sum_{k = \max{\{0,j-l\}}}^{\infty} s_{n,k} (x) \int_0^\infty
s_{n,k+l-j} (t) f^{(l)} (t) dt.
\end{eqnarray*}
\end{proposition}
\begin{remark}
It is easy to see that $f^{(l)}\in E_A$, $A>0$ implies $f \in E_A$.  In the case $A=0$ we have $f \in E_{\tilde{A}}$ for arbitrary $\tilde{A} >0$.
\end{remark}
\begin{remark}
For $j \leq 1$ we get the identity 
\begin{equation*}
(S_{n,j} f)^{(l)} = S_{n,j-l} f^{(l)} .
\end{equation*}
\end{remark}

\begin{proof}
Note that $f^{(l)}\in E_A$ implies $f\in E_A$.

We proceed by induction on $l$.\newline
$l=1$: Using (\ref{eq-der1}), index transform and again (\ref{eq-der1}) we
get 
\begin{eqnarray*}
(S_{n,j}f)^{\prime }(x) & = & n f (0) s_{n,j-1} (x) - n s_{n,j-1} (x)
\int_0^\infty s_{n,0}^{\prime }(t) f (t) dt \\
& & - n \sum_{k = j}^{\infty} s_{n,k} (x) \int_0^\infty s_{n,k+1-j}^{\prime
}(t) f (t) dt \\
& = & n \sum_{k = j-1}^{\infty} s_{n,k} (x) \int_0^\infty s_{n,k+1-j} (t)
f^{\prime }(t) dt \, ,
\end{eqnarray*}
with integration by parts for the last equality. \newline
$l \Rightarrow l+1$ 
\begin{eqnarray*}
(S_{n,j} f)^{(l+1)}(x) & = & n \sum_{\nu =1}^{\min{\{j-1,l-1\}}}
s_{n,j-1-\nu}^{(l-\nu)} (x) f^{(\nu)} (0) \\
& & + n \sum_{k = \max{\{0,j-l\}}}^{\infty} s_{n,k}^{\prime }(x)
\int_0^\infty s_{n,k+l-j} (t) f^{(l)} (t) dt \\
& =:& S_1 + S_2
\end{eqnarray*}
Using (\ref{eq-der1}), index transform and again (\ref{eq-der1}) we get 
\begin{eqnarray*}
S_2 & = & - n \sum_{k = \max{\{0,j-(l+1)\}}}^{\infty} s_{n,k} (x)
\int_0^\infty s_{n,k+l+1-j}^{\prime} f^{ (l)}(t) dt \\
& = & n \sum_{k = \max{\{0,j-(l+1)\}}}^{\infty} s_{n,k} (x) \int_0^\infty
s_{n,k+l+1-j} (t) f^{(l+1)} (t) dt \\
&&+ n s_{n,j-(l+1)} (x) f^{(l)} (0) ,
\end{eqnarray*}
again integration by parts for the last equation. Together with $S_1$ this
proves our proposition.
\end{proof}

For $j\in {\mathbb{Z}}$ define 
\begin{equation}
\mathcal{D}^{2}_j=(1-j)D+e_{1}D^{2},  \label{eq.X2}
\end{equation}%
where $D$ denotes the ordinary differentiation operator.

It is easy to see that 
\begin{equation}  \label{eq.X2-1}
\mathcal{D}^{2}_j =e_{j}D(e_{1-j}D).
\end{equation}
We mention that $\mathcal{D}_j^2 e_j=0$.

\begin{theorem}
\label{lem-com-diff} Let $f\in E$ such that $\mathcal{D}_{j}^{2}f\in E$.
Furthermore we assume that $f^{\prime }(0)=0$ in the case $j\geq 2$. Then
for sufficiently large $n$ the operators $S_{n,j}$ and $\mathcal{D}_{j}^{2}$
commute, i.e., 
\begin{equation*}
\left( S_{n,j}\circ \mathcal{D}_{j}^{2}\right) f=\left( \mathcal{D}%
_{j}^{2}\circ S_{n,j}\right) f.
\end{equation*}
\end{theorem}

\begin{proof}
For 
$j=1$ see \cite[Theorem 3.2]{Heilmann-Tachev-Philipps-2011} and
for $j \leq 1$ \cite[Section 5]{Hei2015} and the references therein. 

Let $j \geq 2$. From Proposition \ref{prop-derivative} we have 
\begin{eqnarray*}
\mathcal{D}^{2}_j ( S_{n,j} f)(x) & = & (1-j)n
\sum_{k=j-1}^{\infty} s_{n,k} (x) \int_0^\infty s_{n,k+1-j} (t) f^{\prime
}(t) dt \\
& & + nx \sum_{k=j-2}^{\infty} s_{n,k} (x) \int_0^\infty s_{n,k+2-j} (t)
f^{\prime \prime }(t) dt .
\end{eqnarray*}
Now we apply the operator $S_{n,j}$ to $\mathcal{D}^{2}_jf$, i.e., 
\begin{eqnarray*}
(S_{n,j} \mathcal{D}^{2}_j f)(x) & = & (1-j)n \sum_{k=j}^{\infty} s_{n,k} (x)
\int_0^\infty s_{n,k-j} (t) f^{\prime }(t) dt \\
& & + n \sum_{k=j}^{\infty} s_{n,k} (x) \int_0^\infty s_{n,k-j} (t) t
f^{\prime \prime }(t) dt \\
& =: & T_1+T_2
\end{eqnarray*}
Using $nt s_{n,k-j} (t) = (k+1-j) s_{n,k+1-j} (t)$ and $k s_{n,k} (x) =nx
s_{n,k-1} (x)$ we get 
\begin{eqnarray*}
T_2 & = & nx \sum_{k=j}^{\infty} s_{n,k-1} (x) \int_0^\infty s_{n,k+1-j} (t)
f^{\prime \prime }(t) dt \\
& & + (1-j) \sum_{k=j}^{\infty} s_{n,k} (x) \int_0^\infty s_{n,k+1-j} (t)
f^{\prime \prime }(t) dt .
\end{eqnarray*}
Index transform in the first sum, integration by parts in the second sum and
application of $s_{n,k+1-j}^{\prime }(t) = n(s_{n,k-j}-s_{n,k+1-j})$ leads
to 
\begin{eqnarray*}
T_2 & = & nx \sum_{k=j-1}^{\infty} s_{n,k} (x) \int_0^\infty s_{n,k+2-j} (t)
f^{\prime \prime }(t) dt \\
& & + (1-j) n \sum_{k=j}^{\infty} s_{n,k} (x) \int_0^\infty s_{n,k+1-j} (t)
f^{\prime }(t) dt \\
& & - (1-j) n \sum_{k=j}^{\infty} s_{n,k} (x) \int_0^\infty s_{n,k-j} (t)
f^{\prime }(t) dt .
\end{eqnarray*}
Thus 
\begin{eqnarray*}
\lefteqn{\mathcal{D}^{2}_j ( S_{n,j}  f)(x) - (S_{n,j} \mathcal{D}^{2}_j f)(x)} \\
& = & (1-j)n s_{n,j-1} (x) \int_0^\infty s_{n,0} (t) f^{\prime }(t) dt +nx
s_{n,j-2} (x) \int_0^\infty s_{n,0} (t)f^{\prime \prime }(t) dt \\
& = & 0
\end{eqnarray*}
as $nx s_{n,j-2} (x) = (j-1) s_{n,j-1} (x)$ and $\int_0^\infty s_{n,0} (t)
f^{\prime \prime }(t) dt = n \int_0^\infty s_{n,0} (t) f^{\prime }(t) dt$.
\end{proof}

Define 
\begin{equation}  \label{eq.X3}
\mathcal{D}^{2l}_j := e_j D^l (e_{l-j} D^l ),\, l\in {\mathbb{N}}.
\end{equation}
For $l=1$ this is in agreement with (\ref{eq.X2}). Moreover, define $D^0_j$
to be the identity operator.

\begin{proposition}
\label{iteratediff} For $l\in {\mathbb{N}}$ the differential operator $%
\mathcal{D}^{2l}_j $ can be written as $l$-th iterate of $\mathcal{D}^{2}_j $%
, i.e., 
\begin{equation}
\mathcal{D}_{j}^{2l}=(\mathcal{D}_{j}^{2})^{l}.  \label{eq.X1}
\end{equation}
\end{proposition}

\begin{proof}
For $j\leq 0$ see \cite[(13)]{Hei2015}. We prove the relation (\ref{eq.X1})
by induction on $l$.

For $l=1$ it is sufficient to use (\ref{eq.X2}).

Suppose (\ref{eq.X1}) is true for $l$. Then 
\begin{eqnarray*}
\mathcal{D}^{2l}_j \circ \mathcal{D}^{2}_j & = & \mathcal{D}^{2l}_j
[(1-j)D+e_1D^2] \\
& = & e_j D^l \left \{ e_{l-j} D^l [(1-j)D +e_1 D^2] \right \} \\
& = & e_j D^l \left \{ e_{l-j} [(l+1-j)D^{l+1} +e_1 D^{l+2}] \right \}
\end{eqnarray*}
and 
\begin{eqnarray*}
\mathcal{D}^{2l+2}_j & = & e_j D^{l+1}
[(l+1-j)e_{l-j}D^{l+1}+e_{l+1-j}D^{l+2}] \\
& = & e_j D^l \left \{ e_{l-j} [(l+1-j)D^{l+1} +e_1 D^{l+2}] \right \} .
\end{eqnarray*}
\end{proof}

As a corollary of Proposition \ref{iteratediff} we derive an explicit
representation of the differential operator.

\begin{corollary}
For $l \in \mathbb{N}$ we have 
\begin{equation*}
{\mathcal{D}}^{2l}_j= ({\mathcal{D}}^{2}_j)^{l} =\sum_{i=0}^{l}\binom{l}{i}%
\left( l-j\right) ^{\underline{l-i}}e_{i}D^{l+i}.
\end{equation*}
\end{corollary}

\begin{proof}
According to (\ref{eq.X3}), we have 
\begin{eqnarray*}
\mathcal{D}^{2l}_j &=&e_{j}D^{l}(e_{l-j}D^{l}) \\
& = & \sum_{i=0}^{l}\binom{l}{i}\left( l-j\right) ^{\underline{l-i}%
}e_{i}D^{l+i}.
\end{eqnarray*}
\end{proof}

Observe that ${\mathcal{D}}^{2l}_j =\sum_{i=j}^{l}\binom{l}{i}%
\left( l-j\right) ^{\underline{l-i}}e_{i}D^{l+i}$ in case of $l \geq j \geq 0$.

\section{Eigenstructure of $S_{n,j}$}

In this section we study the eigenstructure of the operators $S_{n,j}$, $j\in \mathbb{Z}$.
As we already know, $e_0$ is an eigenfunction with eigenvalue $1$ for each $j\in \mathbb{Z}$. Now we present a family of eigenfunctions $g_{j,p}$, $p\in 
\mathbb{R} $, of $S_{n,j}$ which are independent of $n$ and can be expressed as  
modified Bessel functions $I_{j}$
of the first kind. 

\begin{theorem}
\label{theorem-eigenfunctions} Let $j \in \mathbb{Z}$. For each $p\in {%
\mathbb{R}}$, the function 
\begin{align*}
g_{j,p}(x) &=\sum_{m=\max{\{0,-j\}}}^{\infty} \frac{p^m x^{m+j}}{m!(m+j)!}
\end{align*}%
is an eigenfunction of $S_{n,j}$ with eigenvalue $e^{p/n}$.
\end{theorem}
\begin{proof}
For $p=0$ and $j \geq 0$ we get $g_{j,0}(x) = \frac{x^j}{j!} = \frac{1}{j!} S_{n,j} (e_j;x)$ by
the preservation of $e_j$.

Now let $p\not= 0$.

By using (\ref{eq-int1}) we get for $j \in \mathbb{N}_0$
\begin{eqnarray*}
	(S_{n,j} g_{j,p})(x)
	& = &
	\sum_{k=j}^{\infty} s_{n,k} (x) \sum_{m=0}^{\infty} 
	\frac{\left(\frac{p}{n}\right)^m}{m!(m+j)!} \cdot \frac{(k+m)!}{(k-j)!}n^{-j}
\\
	& = &
	e^{-nx} \sum_{m=0}^{\infty} \frac{\left(\frac{p}{n}\right)^m}{m!(m+j)!} 
	\sum_{k=0}^{\infty} \frac{n^k x^{k+j}}{(k+j)!} \cdot \frac{(k+m+j)!}{k!} .
\end{eqnarray*}
As
\begin{eqnarray*}
	\sum_{k=0}^{\infty} \frac{n^k x^{k+j}}{(k+j)!} \cdot \frac{(k+m+j)!}{k!}
	& = &
	\sum_{k=0}^{\infty} \frac{n^k}{k!} \left ( x^{k+m+j} \right )^{(m)}
\\
	 = 
	\frac{d^m}{dx^m} \left (x^{m+j} e^{nx} \right )
	& = &
	\sum_{l=0}^{m} \binom{m}{l} \frac{(m+j)!}{(l+j)!} x^{l+j} n^l e^{nx} ,
\end{eqnarray*}
we derive
\begin{eqnarray*}
	(S_{n,j} g_{j,p})(x)
	& = &
	\sum_{m=0}^{\infty} \left(\frac{p}{n}\right)^m
	\sum_{l=0}^{m} \frac{n^l x^{l+j}}{l! (l+j)! (m-l)!}
\\
	& = &
	\sum_{l=0}^{\infty}  \frac{x^{l+j}p^l}{l! (l+j)!}
	\sum_{m=l}^{\infty} \frac{ \left(\frac{p}{n}\right)^{m-l}}{(m-l)!}
\\
	& = &
	e^{\frac{p}{n}} g_{j,p} (x) .
\end{eqnarray*}	
 For $-j \in \mathbb{N}$ we get in a similar way
\begin{eqnarray*}
	(S_{n,j} g_{j,p})(x)
	& = &
	e^{-nx} \sum_{m=-j}^{\infty} \frac{\left(\frac{p}{n}\right)^m n^{-j}x^j}{m!(m+j)!} 
	\sum_{k=0}^{\infty} \frac{n^k x^{k-j}}{(k-j)!} \cdot \frac{(k+m)!}{k!} .
\end{eqnarray*}
As
\begin{eqnarray*}
	\sum_{k=0}^{\infty} \frac{n^k x^{k-j}}{(k-j)!} \cdot \frac{(k+m)!}{k!}
	& = &
	\sum_{k=0}^{\infty} \frac{n^k}{k!} \left ( x^{k+m} \right )^{(m+j)}
\\
	 = 
	\frac{d^{m+j}}{dx^{m+j}} \left (x^{m} e^{nx} \right )
	& = &
	\sum_{l=0}^{m+j} \binom{m+j}{l} \frac{m!}{(l-j)!} x^{l-j} n^l e^{nx} ,
\end{eqnarray*}
we conclude
\begin{eqnarray*}
	S_{n,j} (g_{j,p};x)
	& = &
	\sum_{m=-j}^{\infty} \left(\frac{p}{n}\right)^m
	\sum_{l=0}^{m+j} \frac{n^{l-j} x^{l}}{l! (l-j)! (m+j-l)!}
\\
	& = &
	\sum_{l=0}^{\infty}  \frac{x^{l}p^{l-j}}{l! (l-j)!}
	\sum_{m=l-j}^{\infty} \frac{ \left(\frac{p}{n}\right)^{m+j-l}}{(m+j-l)!}
\\
	& = &
	e^{\frac{p}{n}} g_{j,p} (x) .
\end{eqnarray*}	
\end{proof}
\begin{remark}
For $p \not= 0$ the eigenfunctions can be written in terms of modified Bessel functions of the first kind,
i.e.,
\begin{equation}
	g_{j,p} (x)  = (px)^{\frac{j}{2}} I_j (2 \sqrt{px}) .
\end{equation}
\end{remark}

\section{Eigenstructure of $\mathcal{D}_{j}^{2}$}

The Bessel differential equation is given by 
\begin{equation}
z^{2}y^{\prime \prime }+zy^{\prime }+\left( z^{2}-\nu ^{2}\right) y=0.
\label{ODE-Bessel}
\end{equation}%
Solutions are the Bessel functions of the first kind $J_{\nu }\left(
z\right) $ and of the second kind $Y_{\nu }\left( z\right) $ (also called
Weber's function). Each is a holomorphic function of $z$ throughout the $z$%
-plane cut along the negative real axis. $J_{\nu }\left( z\right) $ and $%
J_{-\nu }\left( z\right) $ are linearly independent except when $\nu $ is an
integer. In our case we consider the integer $\nu =j$. Therefore, we take as
a fundamental system the functions $J_{\nu }\left( z\right) $ and $Y_{\nu
}\left( z\right) $ which are linearly independent for all values of $\nu $.
For integer $\nu $, $J_{\nu }$ is an entire function with series expansion 
\begin{equation*}
J_{\nu }\left( z\right) =\left( -1\right) ^{\nu }J_{-\nu }\left( z\right)
=\left( \frac{z}{2}\right) ^{\nu }\sum_{k=0}^{\infty }\frac{\left( -\frac{1}{%
4}z^{2}\right) ^{k}}{k!\Gamma \left( \nu +k+1\right) }\text{ \qquad }\left(
z\in C,\text{ }\nu =0,1,2,\ldots \right) .
\end{equation*}%
$Y_{\nu }\left( z\right) $ has a more complicated representation.

We show that the solutions of (\ref{ODE-Bessel}) are closely connected to
the eigenfunctions of the differential operator $\mathcal{D}%
_{j}^{2}=(1-j)D+e_{1}D^{2}$ as defined in (\ref{eq.X2}). Since the second
order linear ODE $\mathcal{D}_{j}^{2}y=0$ has the both independent solutions 
$e_{0},e_{j}$, the kernel of $\mathcal{D}_{j}^{2}$ is the span of $\left\{
e_{0},e_{j}\right\} $. The case $p\neq 0$ will be considered in the next
theorem. 

\begin{theorem}
\label{theorem-Bessel-ODE-EF-Dj2}Let $j\in \mathbb{Z}$ and $p\in \mathbb{%
R\smallsetminus }\left\{ 0\right\} $. If $y\left( x\right) $ is a solution
of the Bessel differential equation 
\begin{equation}
x^{2}y^{\prime \prime }+xy^{\prime }+\left( x^{2}-j^{2}\right) y=0
\label{ODE-Bessel-with-j}
\end{equation}%
on $\left( 0,\infty \right) $, then the function $h\left( x\right) =g\left(
px\right) $, where 
\begin{equation*}
g\left( x\right) =x^{j/2}y\left( 2i\sqrt{x}\right)
\end{equation*}%
is an eigenfunction of $\mathcal{D}_{j}^{2}$ to the eigenvalue $p$.
\end{theorem}

\begin{remark}
 If we take $y\left( x\right) =J_{j}\left( x\right) $
with $j,x\geq 0$, we have 
\begin{eqnarray*}
	h\left( x\right) 
	= 
	\left( px\right) ^{j/2}J_{j}\left( 2i\sqrt{px}\right)
	& = &\left( px\right) ^{j/2}\left( i\sqrt{px}\right) ^{j}\sum_{k=0}^{\infty }
	\frac{\left( px\right) ^{k}}{k!\Gamma \left( j+k+1\right) }
\\
	& = &i^{j}\left(px\right) ^{j}\sum_{k=0}^{\infty }
	\frac{\left( px\right) ^{k}}{k!\Gamma \left( j+k+1\right) }.
\end{eqnarray*}%
For $p\geq 0$, we essentially have the modified Bessel function $I_{j}$. To
avoid this, we can replace $p$ with $-p$. For negative integers $j$ we can
take advantage of the relation $J_{j}\left( x\right) =\left( -1\right)
^{j}J_{-j}\left( x\right) $.
\end{remark}

\begin{proof}[Proof of Theorem~\protect\ref{theorem-Bessel-ODE-EF-Dj2}]
Direct calculation yields 
\begin{eqnarray*}
g^{\prime }\left( x\right) &=&\frac{j}{2}x^{j/2-1}y\left( 2i\sqrt{x}\right)
+ix^{\left( j-1\right) /2}y^{\prime }\left( 2i\sqrt{x}\right) , \\
g^{\prime \prime }\left( x\right) &=&\frac{j}{2}\frac{j-2}{2}%
x^{j/2-2}y\left( 2i\sqrt{x}\right) +\frac{j}{2}x^{j/2-1}ix^{-1/2}y^{\prime
}\left( 2i\sqrt{x}\right) \\
&&+i\frac{j-1}{2}x^{\left( j-3\right) /2}y^{\prime }\left( 2i\sqrt{x}\right)
+i^{2}x^{\left( j-2\right) /2}y^{\prime \prime }\left( 2i\sqrt{x}\right) \\
&=&\frac{j}{2}\frac{j-2}{2}x^{j/2-2}y\left( 2i\sqrt{x}\right) +i\frac{2j-1}{2%
}x^{\left( j-3\right) /2}y^{\prime }\left( 2i\sqrt{x}\right) -x^{\left(
j-2\right) /2}y^{\prime \prime }\left( 2i\sqrt{x}\right) .
\end{eqnarray*}%
By the Bessel differential equation (\ref{ODE-Bessel-with-j}), we have 
\begin{equation*}
y^{\prime \prime }\left( 2i\sqrt{x}\right) =-\frac{1}{2i\sqrt{x}}y^{\prime
}\left( 2i\sqrt{x}\right) -\left( 1+\frac{j^{2}}{4x}\right) y\left( 2i\sqrt{x%
}\right)
\end{equation*}%
and it follows that 
\begin{eqnarray*}
g^{\prime \prime }\left( x\right) &=&\frac{j^{2}-2j}{4}x^{j/2-2}y\left( 2i%
\sqrt{x}\right) +i\frac{2j-1}{2}x^{\left( j-3\right) /2}y^{\prime }\left( 2i%
\sqrt{x}\right) \\
&&-\frac{i}{2}x^{\left( j-3\right) /2}y^{\prime }\left( 2i\sqrt{x}\right)
+\left( x^{j/2-1}+\frac{j^{2}}{4}x^{j/2-2}\right) y\left( 2i\sqrt{x}\right)
\\
&=&\left( x^{j/2-1}+\frac{j^{2}-j}{2}x^{j/2-2}\right) y\left( 2i\sqrt{x}%
\right) +i\left( j-1\right) x^{\left( j-3\right) /2}y^{\prime }\left( 2i%
\sqrt{x}\right) .
\end{eqnarray*}%
Hence, 
\begin{eqnarray*}
&&p^{-1}\left( p^{2}xg^{\prime \prime }\left( px\right) +\left( 1-j\right)
pg^{\prime }\left( px\right) \right) \\
&=&\left( \left( px\right) ^{j/2}+\frac{j^{2}-j}{2}\left( px\right)
^{j/2-1}\right) y\left( 2i\sqrt{px}\right) +i\left( j-1\right) \left(
px\right) ^{\left( j-1\right) /2}y^{\prime }\left( 2i\sqrt{px}\right) \\
&&+\frac{j\left( 1-j\right) }{2}\left( px\right) ^{j/2-1}y\left( 2i\sqrt{px}%
\right) +i\left( 1-j\right) \left( px\right) ^{\left( j-1\right)
/2}y^{\prime }\left( 2i\sqrt{px}\right) \\
&=&\left( \left( px\right) ^{j/2}+\frac{j^{2}-j}{2}\left( px\right)
^{j/2-1}\right) y\left( 2i\sqrt{px}\right) +\frac{j\left( 1-j\right) }{2}%
\left( px\right) ^{j/2-1}y\left( 2i\sqrt{px}\right) \\
&=&\left( px\right) ^{j/2}y\left( 2i\sqrt{px}\right) =g\left( px\right) ,
\end{eqnarray*}%
i.e., the function $h\left( x\right) =g\left( px\right) $ is a solution of
the ODE\ 
\begin{equation*}
xh^{\prime \prime }+\left( 1-j\right) h^{\prime }=ph.
\end{equation*}%
Noting definition (\ref{eq.X2}) concludes the proof.
\end{proof}

The next theorem shows a connection between eigenfunctions of $S_{n,j}$ and
eigenfunctions of $\mathcal{D}_{j}^{2}$.

\begin{theorem}
Let $j\in \mathbb{Z}$ and $p\in \mathbb{R}$. If a function $h\in C^{2}\left(
0,\infty \right) $ is an eigenfunction of $S_{n,j}$ to the eigenvalue $%
e^{p/n}$ for infinitely many integers $n$, then $h$ is an eigenfunction of $%
\mathcal{D}_{j}^{2}$ to the eigenvalue $p$.
\end{theorem}

\begin{proof}
Since $h\in C^{2}\left( 0,\infty \right) $, we have for each $x>0$ the Voronovskaja
relation 
\begin{equation*}
\left( S_{n,j}h\right) \left( x\right) =h\left( x\right) +n^{-1}\left( 
\mathcal{D}_{j}^{2}h\right) \left( x\right) +o\left( 1/n\right) \text{
\qquad }\left( n\rightarrow \infty \right) .
\end{equation*}%
On the other hand, we have 
\begin{equation*}
\left( S_{n,j}h\right) \left( x\right) =e^{p/n}h\left( x\right) =\left( 1+%
\frac{p}{n}+o\left( \frac{1}{n}\right) \right) h\left( x\right) \text{ \qquad }%
\left( n\rightarrow \infty \right) .
\end{equation*}%
Here $n\rightarrow \infty $ means that we take only such $n$ for which $%
S_{n,j}h=e^{p/n}h$. Comparing both relations we infer that 
\begin{equation*}
\left( \mathcal{D}_{j}^{2}h\right) \left( x\right) =ph\left( x\right)
+o\left( 1\right) \text{ \qquad }\left( n\rightarrow \infty \right) ,
\end{equation*}%
which completes the proof.
\end{proof}

The converse problem is more complicated. It is easy to show the following
result.

\begin{theorem}
If a function $h\in C^{2}\left( 0,\infty \right) \cap E$ is an eigenfunction
of $\mathcal{D}_{j}^{2}$ to the eigenvalue $p$ on $\left( 0,\infty \right) $%
, it follows, for each $x>0$ and for each constant $c>0$, the asymptotic
relation 
\begin{equation*}
\left( S_{n,j}h\right) \left( x\right) =e^{p/n}h\left( x\right) +o\left(
n^{-c}\right) \text{ \qquad }\left( n\rightarrow \infty \right) .
\end{equation*}
\end{theorem}

\begin{proof}
The fact, that $h$ is an eigenfunction of $\mathcal{D}_{j}^{2}$ implies that 
$h$ is infinitely often differentiable. Let $q$ be a positive integer. By \cite[Theorem 25]{Snj2}
 the operators $S_{n,j}$ possess the asymptotic
expansion 
\begin{equation*}
\left( S_{n,j}h\right) \left( x\right) =h\left( x\right) +\sum_{k=1}^{q}%
\frac{1}{k!n^{k}}\left( {\mathcal{D}}_{j}^{2k}h\right) \left( x\right)
+o\left( n^{-q}\right) \text{ \qquad }\left( n\rightarrow \infty \right) .
\end{equation*}%
Observing that ${\mathcal{D}}_{j}^{2k}h= \left ({\mathcal{D}}_{j} \right )^k h =p^{k}h$ (see Proposition \ref{iteratediff}) and 
\begin{equation*}
\sum_{k=0}^{q}\frac{p^{k}}{k!n^{k}}=e^{p/n}+O\left( n^{-q-1}\right) \text{
\qquad }\left( n\rightarrow \infty \right)
\end{equation*}%
completes the proof.
\end{proof}


\begin{thebibliography}{99}
\bibitem{Snj2} U.~Abel, A.~M. Acu, M.~Heilmann, I.~Ra\c{s}a,
Asymptotic properties for a general class of
Sz\'asz-Mirakjan-Durrmeyer operators,  	arXiv:2407.16474.

\bibitem{AbIv2005} U.~Abel, M.~Ivan, Enhanced asymptotic approximation and
approximation of truncated functions by linear operators, Constructive
Theory of Functions, Proc. Int. Conf. CTF, Varna, June 2 - June 6, 2005,
B.~D.~Bojanov (Ed.), Prof. Marin Drinov Academic Publishing House, 2006,
1--10.


\bibitem{Abramowitz} M.~Abramowitz, I.~A.~Stegun, (Eds.): Handbook of
Mathematical Functions with Formulas, Graphs and Mathematical Tables,
National Bureau of Standards Applied Mathematics Series 55, Issued June
1964, Tenth Printing, December 1972, with corrections.



\bibitem{He1988} M.~Heilmann, Commutativity of operators from
Baskakov-Durrmeyer type, Constructive theory of functions (Varna, 1987),
Publ. House Bulgar. Acad. Sci., Sofia, pp. 197--206, 1988.

\bibitem{Hei1989} M.~Heilmann, Direct and converse results for operators of
Baskakov-Durrmeyer type, Approximation Theory Appl. 5 (1):105--127 (1989).

\bibitem{He1992} M.~Heilmann, Erh\"{o}hung der Konvergenzgeschwindigkeit bei
der Approximation von Funktionen mit Hilfe von Linearkombinationen
spezieller positiver linearer Operatoren. Habilitationschrift Universit\"{a}%
t Dortmund, 1991.

\bibitem{Hei2015} M.~Heilmann, Commutativity and spectral properties of
genuine Baskakov-Durrmeyer type operators and their $k$th order Kantorovich
modification, J. Numer. Anal. Approx. Theory, 44, No.2, 166--179 (2015).

\bibitem{HeiMue1989} M.~Heilmann, M. W. M\"uller, On simultaneous
approximation by the method of Baskakov-Durrmeyer operators, Numer. Funct.
Anal. Optim. 10 (1989), 127--138.

\bibitem{Heilmann-Tachev-Philipps-2011} M.~Heilmann, G.~Tachev,
Commutativity, direct and strong converse results for Philipps operators,
East J. Approx. 17 (3), 299--317 (2011).

\bibitem{Heilmann-Tachev-Phillips-2013} M.~Heilmann, G.~Tachev, Linear
combinations of genuine Sz\'{a}sz-Mirakjan-Durrmeyer operators, Advances in
Applied Mathematics and Approximation - Contributions from AMAT 2012, Band
41, Seite 85-106, Applied Mathematics and Approximation Theory 2012, Ankara,
Turkey. May 2012, In Anastassiou, George A. and Duman, Oktay, Editor,
Herausgeber: Springer Proceedings in Mathematics and Statistics, 2013


\bibitem{MaTo1985} S.~M.~Mazhar, V.~Totik, Approximation by modified Sz\'asz
operators, {Acta Sci. Math.} 49, 257--269 (1985).

\bibitem{Ph1954} R.~S.~Phillips, An inversion formula for Laplace transforms
and semi-groups of linear operators, Ann. Math. (2) 59 (1954), 325--356.

\bibitem{Sw1} T.~\'Swiderski, On certain properties of the combinations of
Sz\'asz-Durrmeyer operators, Anal. Theory Appl. Vol. 27(2) (2011), 167--180.
\end{thebibliography}
\end{document}